\documentclass[11pt]{article}
\usepackage{amsmath,amsthm,amssymb,ifsym,a4wide} 
\usepackage[twoside]{geometry}
\usepackage{graphicx}
\usepackage{epsfig} 
\usepackage{slashed}  
\usepackage[hidelinks]{hyperref}
\newtheorem{theorem}{Theorem}[section]
\newtheorem{proposition}[theorem]{Proposition}
\newtheorem{lemma}[theorem]{Lemma}

\numberwithin{equation}{section} 
\numberwithin{figure}{section}  
\usepackage{amssymb, amscd, mathrsfs, wasysym} 

\newcommand \la \langle
\newcommand \ra \rangle

\newcommand \Kcal {\mathcal{K}}
\newcommand \Hcal {\mathcal{H}}

\newcommand \underdel {\underline \partial}

\newcommand \trianglerightNEW \triangleright

\newcommand \auth {\textsc} 

\newcommand \CC {\mathbb C}

\newcommand \bei {\begin{itemize}}
\newcommand \eei {\end{itemize}}
\newcommand \be {\begin{equation}}
\newcommand \bel {\be\label}
\newcommand \ee {\end{equation}}
\newcommand \del \partial
\newcommand \RR {\mathbb R}

\newcommand \eps \epsilon

\usepackage[most]{tcolorbox} 

\let\oldmarginpar\marginpar
\renewcommand\marginpar[1]{\-\oldmarginpar[\raggedleft\footnotesize #1]%
{\raggedright\footnotesize #1}}
\begin{document}

\title{\bf \Large 
Asymptotic Behavior of the Solution to the Klein-Gordon-Zakharov Model in Dimension Two}

\author{Shijie Dong\footnote{Fudan University, School of Mathematical Sciences, Shanghai, China.
Email: dongs@ljll.math.upmc.fr.
}
}


\date{\today}
\maketitle

\begin{abstract} 
{
Consider the Klein-Gordon-Zakharov equations in $\RR^{1+2}$, and we are interested in establishing the small global solution to the equations and in investigating the pointwise asymptotic behavior of the solution. The Klein-Gordon-Zakharov equations can be regarded as a coupled semilinear wave and Klein-Gordon system with quadratic nonlinearities which do not satisfy the null conditions, and the fact that wave components and Klein-Gordon components do not decay sufficiently fast makes it harder to conduct the analysis. In order to conquer the difficulties, we will rely on the hyperboloidal foliation method and a minor variance of the ghost weight method. As a side result of the analysis, we are also able to show the small data global existence result for a class of quasilinear wave and Klein-Gordon system violating the null conditions.
}
\end{abstract}

\tableofcontents

\section{Introduction} 

\paragraph{Background and motivations}

The study of nonlinear wave equations, nonlinear Klein-Gordon equations, and their coupled systems has attracted the attention from the researchers since decades ago. Among the most concerned questions are that what kinds of quadratic nonlinearities can guarantee the small data global existence results for the equations, and that what is the asymptotic behavior of the global solution. The milestone work on this subject is due to Christodoulou \cite{Christodoulou} and Klainerman \cite{Klainerman86} for nonlinear wave equations in $\RR^{1+3}$, and is due to Klainerman \cite{Klainerman85} and Shatah \cite{Shatah} for nonlinear Klein-Gordon equations in $\RR^{1+3}$. After that, consecutive progress is made in this field, from different types of nonlinearities to various dimensions. 


We now briefly revisit some existing results on the coupled wave and Klein-Gordon equations, with or without physical models behind, which are the main motivation of this paper. The study of the coupled wave and Klein-Gordon equations has its own difficulties, which is due to the well-known fact that the scaling vector field $L_0 = t\del_t + x^a \del_a$ does not commute with the Klein-Gordon operator as well as many others. 
On one hand, in $\RR^{1+3}$, the pioneering work is due to Bachelot \cite{Bachelot} on the Dirac-Klein-Gordon system and Georgiev \cite{Georgiev} on the general wave and Klein-Gordon system with null nonlinearities. Later on, lots of other kinds of coupled wave and Klein-Gordon equations were studied, see for instance \cite{OTT, PsarelliA, PsarelliT, Tsutsumi, Katayama12a, Katayama18, PLF-YM-book, PLF-YM-cmp, Ionescu-P, Ionescu-P2, Wang, Klainerman-QW-SY, Fang, Dong1811, DLW, Dong1912}.
On the other hand, in $\RR^{1+2}$, the slow decay nature of the wave components and the Klein-Gordon components makes the study of the coupled wave and Klein-Gordon equations more formidable. The first such result is due to Ma \cite{MaH1, MaH2}, where a class of quasilinear wave and Klein-Gordon system was shown to admit the global solution. Shortly after that, more types quadratic nonlinearities were treated, see \cite{Ma17, Ma19,Stingo18, Stingo19, Dong2005}.

Next, we return to the Klein-Gordon-Zakharov equations, which appear in plasma physics, see for instance \cite{Zakharov, Dendy, Masmoudi}. The first small data global existence result was established by Ozawa, Tsutaya, and Tsutsumi \cite{OTT} in $\RR^{1+3}$, and then some other alternative proofs were provided in \cite{Tsutaya, Katayama12a, Dong1811, Dong1905}. Later on, Ozawa, Tsutaya, and Tsutsumi \cite{OTT2} showed that the Klein-Gordon-Zakharov equations are well-posed with different propagation speeds in $\RR^{1+3}$. Besides, there exists other interesting work on different aspects of the studies on Klein-Gordon-Zakharov equations: Guo and Yuan \cite{Guo} proved that the Klein-Gordon-Zakharov equations with first order expressions admit the global smooth solution in $\RR^{1+2}$ when the initial data are bounded;
it is also worth to mention the interesting work of Masmoudi and Nakanishi \cite{Masmoudi} in which the authors studied the convergence of the Klein–Gordon–Zakharov equations to the Schrodinger equation as some parameters go to $+\infty$.

We recall, on one hand, that one of the successful tools in dealing with coupled wave and Klein-Gordon equations is the hyperboloidal foliation method, dating back to Klainerman \cite{Klainerman86} and Hormander \cite{Hormander}, which can be regarded as the Klainerman's vector field method on hyperboloids, and which was developed by LeFloch-Ma \cite{PLF-YM-book, PLF-YM-cmp, PLF-YM-arXiv1, PLF-YM-arXiv2}, and by Klainerman-Wang-Yang \cite{Wang, Klainerman-QW-SY}. In this method, the use of the scaling vector field $L_0$ is avoided, which makes it consistent with both the wave equations and the Klein-Gordon equations. On the other hand, the ghost weight method, introduced by Alinhac \cite{Alinhac1, Alinhac2}, is very powerful in studying nonlinear partial differential equations in various dimensions. Originally the method was used on the wave equations in $\RR^{1+2}$, but it can also be use to study the coupled wave and Klein-Gordon systems, see for instance \cite{Dong2005}.

Motivated by the existing work on the coupled wave and Klein-Gordon equations, we are interested in establishing the small data global existence result for the Klein-Gordon-Zakharov equations in $\RR^{1+2}$, and in addition, we will also demonstrate the asymptotic behavior of the solution. In order to overcome the slow decay property of the wave and the Klein-Gordon components in dimension two, we will combine the ghost weight method and the hyperboloidal foliation method. As a consequence, we can generalise, to some extent, the result on the quasilinear wave and Klein-Gordon equations in $\RR^{1+2}$ of \cite{Stingo18}.

\paragraph{Model problem and main difficulties}

We will consider the Klein-Gordon-Zakharov model in $\RR^{1+2}$
\bel{eq:model-KGZ}
\aligned
-\Box E + E = - n E,
\\
-\Box n = \Delta |E|^2.
\endaligned
\ee
The unknowns are $E = (E^1, E^2)$ taking values in\footnote{Originally $E$ takes values in $\RR^2$, but more general cases of taking values in $\CC^{N_0}$ with $N_0 = 1, 2, \cdots$ can also be treated.} $\RR^2$, $n$ taking values in $\RR$, which are the Klein-Gordon component and the wave component respectively.
We take the signature $(-, +, +)$ in the spacetime $\RR^{1+2}$. As usual, $\Box = \del_\alpha \del^\alpha$ represents the wave operator, $\Delta = \del_a \del^a$ represents the Laplace operator, with the Greek letters $\alpha, \beta, \cdots \in \{0, 1, 2\}$ denoting the spacetime indices, and Latin letters $a, b, \cdots \in \{1, 2\}$ representing the space indices, and the Einstein summation convention is adopted unless otherwise specified.

The initial data are prescribed at the slice $t=t_0 =2$
\bel{eq:model-ID}
\aligned
\big( E, \del_t E \big) (t_0, \cdot)
=
\big( E_0, E_1 \big),
\qquad
\big( n, \del_t n \big) (t_0, \cdot)
=
\big( n_0, n_1 \big) := \big( \Delta n^{\Delta}_0, \Delta n^{\Delta}_1 \big),
\endaligned
\ee
and the functions $(E_0, E_1, n^\Delta_0, n^\Delta_1)$ are assumed to be supported in the unit ball $\{ x : |x| \leq 1 \}$.

Due to the serious problem that wave components and Klein-Gordon components decay insufficiently fast in $\RR^{1+2}$, the quadratic nonlinearities appearing in \eqref{eq:model-KGZ}, which violate the null conditions, are at the border line of integrability. To be more precise, in $\RR^{1+2}$ linear waves decay at the speed of $t^{-1/2}$, while linear Klein-Gordon components decay at the speed of $t^{-1}$. This means that the best we can expect for the nonlinearities is
$$
\big\| n E \big\|_{L^2(\RR^2)} 
\lesssim t^{-1},
\qquad
\big\| \Delta |E|^2 \big\|_{L^2(\RR^2)} 
\lesssim t^{-1},
$$
which are non-integrable quantities. Thus under this situation, it is very hard to prove the sharp pointwise decay results,  as well as closing the bootstrap, of $E$ and $n$.

Recall that in the framework of the hyperboloidal foliation method, we will integrate over the hyperbolic time $s = \sqrt{t^2-|x|^2}$, which means we need to show $|n| \lesssim s^{-1}$ to close the bootstrap argument. But this does not seem to be easy, because when we go to the $n$ equation, the fact that the Klein-Gordon component $E$ decays only like $|E| \lesssim t^{-1}$, as already remarked, leads to a polynomial growth even in the natural energy of the $n$ component. 

Besides, the lack of partial derivatives on the $n$ component in the $E$ equation seems also to be a problem (which is difficult to handle even in $\RR^{1+3}$, see for instance \cite{PLF-YM-cmp, Wang, Ionescu-P, Ionescu-P2, Klainerman-QW-SY, Fang, Dong1811, Dong1912}), which is meanly due to the fact that the $L^2$--type norms of $n$ cannot be bounded by its natural wave energy.

\paragraph{Main theorems}

We first introduce the natural energy for the wave components and the Klein-Gordon components in the hyperboloidal foliation setting in $\RR^{1+2}$. 
Let $\phi = \phi(t, x)$ be a sufficiently nice function supported in the spacetime region $\{ (t, x) : t \geq |x| + 1 \}$, then its natural energy on the hyperboloid $\Hcal_s = \{ (t, x) : t^2 = |x|^2 + s^2 \}$ is defined by
\bel{eq:E1}
E_m (s, \phi)
:=
\int_{\Hcal_s} (\del_t \phi)^2 + \sum_a (\del_a \phi)^2 + 2 (x^a/t) \del_a \phi \del_t \phi + m^2 \phi^2 \, dx.
\ee
The abbreviation $E (s, \phi) = E_0 (s, \phi)$ is used.

Recall that our goal is to prove the small data global existence result to the model problem \eqref{eq:model-KGZ}, and to derive the pointwise decay result of the solution. Now the first main result is illustrated.

\begin{theorem}\label{thm:main1}
Consider the Klein-Gordon-Zakharov equations in \eqref{eq:model-KGZ}, and let $N \geq 11$ be an integer. There exits $\eps_0 >0$, such that for all $\eps < \eps_0$, and all compactly supported initial data satisfying the smallness condition
\be 
\| E_0 \|_{H^{N+2}} + \| E_1 \|_{H^{N+1}}
+
\| n^\Delta \|_{H^{N+3}} + \| n^\Delta \|_{H^{N+2}}
\leq \eps,
\ee
the Cauchy problem \eqref{eq:model-KGZ}--\eqref{eq:model-ID} admits a global-in-time solution $(E, n)$, 
which satisfies the following sharp pointwise decay results
\bel{eq:thm-decay}
|E(t, x)| \lesssim t^{-1},
\quad
|n(t, x) | \lesssim t^{-1/2} (t-r)^{-1/2}.
\ee
Furthermore, with the notation for the Lorentz boosts $L_a = x_a \del_t + t \del_a$ and $0<\delta<1/24$, the following energy estimates are also valid
\bel{eq:thm-E}
\aligned
E_1 (s, \del^I L^J E)^{1/2}
+
E_1 (s, \del \del^I L^J E)^{1/2}
&\lesssim s^\delta,
\qquad
&|I| + |J| \leq N,
\\
E_1 (s, \del^I L^J E)^{1/2}
&\lesssim 1,
\qquad
&|I| + |J| \leq N-3,
\\
E (s, \del^I L^J n)^{1/2}
&\lesssim s^\delta,
\qquad
&|I| + |J| \leq N.
\endaligned
\ee

\end{theorem}

Although the system \eqref{eq:model-KGZ} is with critical nonlinearities, the sharp pointwise decay results of the solution $(E, n)$ can still be obtained.
And this is the first global existence and asymptotic results, as far as we know, on the wave and Klein-Gordon systems in $\RR^{1+2}$ violating the null conditions.

We note that the divergence form nonlinearities, see for instance \cite{Katayama12a, Dong1811, Dong2005}, in the $n$ equation is a gain in some sense, and we now take advantage of this structure.
Expressing the original equations in \eqref{eq:model-KGZ} in terms of the scalar valued variables $(E^a, n)$ with $a = 1, 2$ gives us
\bel{eq:model-KGZ2}
\aligned
-\Box E^a + E^a = - n E^a,
\\
-\Box n = \Delta \big( (E^1)^2 + (E^2)^2 \big).
\endaligned
\ee
Then we proceed to introduce the new variable $(n^\Delta)$ with the relation
\bel{eq:decomp}
n = \Delta n^\Delta,
\ee
and this new variable satisfies the following wave equation:
\bel{eq:n-12}
-\Box n^\Delta = (E^1)^2 + (E^2)^2,
\qquad
\big( n^\Delta, \del_t n^\Delta  \big) (t_0)
= (n^\Delta_0, n^\Delta_1).
\ee
To sum up, we will consider, in the analysis, the scalar valued variables
\be 
E^a, \qquad n^\Delta,
\ee
which are related to the original unknowns $(E, n)$ by the relations
\be 
E = (E^1, E^2),
\qquad
n = \Delta n^\Delta.
\ee
And the variables $(E^a, n^\Delta)$ are solutions to the following equations
\bel{eq:model-KGZ2}
\aligned
&-\Box E^a + E^a = \Delta n^\Delta E^a,
\qquad
&\big( E^a, \del_t E^a \big)(t_0)
= (E^a_0, E^a_1),
\\
&-\Box n^\Delta = (E^1)^2 + (E^2)^2,
\qquad
&\big( n^\Delta, \del_t n^\Delta  \big) (t_0)
= (n^\Delta_0, n^\Delta_1).
\endaligned
\ee

To achieve the sharp pointwise decay of $n = \Delta n^\Delta$, we prove a new type of energy estimates as well as a new type of Sobolev--type inequality accordingly, which is inspired by the ghost weight method. This energy estimates allow us to gain $t$ decay at the expense of losing $t-|x|$ decay, and fortunately the loss of $t-r$ decay is not a problem as $\del \del n^\Delta$, roughly speaking, has an extra $(t-|x|)^{-1}$ decay compared to $\del L n^\Delta, \del n^\Delta$ (see \eqref{eq:ddu}), for instance.
On the other hand, in order to obtain the sharp pointwise decay estimates of $E$, we prove the uniform energy estimates, which is thanks to the trick that we move the good factor $s'/t$ in the energy estimates \eqref{eq:w-EE2} to the source term, and the details are demonstrated in the proof of Proposition \ref{prop:refine}.

Inspired by the treatment on the Klein-Gordon-Zakharov equations in the form \eqref{eq:model-KGZ2}, we find that our method can also be applied to the following quasilinear wave and Klein-Gordon system:
\bel{eq:model-q} 
\aligned
&-\Box v + v + P_1^{\alpha \beta} v \del_\alpha \del_\beta u + P_2^{\alpha \beta \gamma} \del_\gamma v \del_\alpha \del_\beta w = 0,
\\
&-\Box w + P_1^{\alpha \beta} v \del_\alpha \del_\beta v + P_2^{\alpha \beta \gamma} \del_\gamma v \del_\alpha \del_\beta v = 0,
\\
&\big( v, \del_t v, w, \del_t w \big)(t_0) = (v_0, v_1, w_0, w_1).
\endaligned
\ee 
In the coupled system \eqref{eq:model-q}, $P_1^{\alpha \beta}, P_2^{\alpha \beta \gamma}$ are constants, which do not need to satisfy the null conditions. As a comparison, in \cite{Stingo18} the nonlinearities $P_2^{\alpha \beta \gamma} \del_\gamma v \del_\alpha \del_\beta u, P_2^{\alpha \beta \gamma} \del_\gamma v \del_\alpha \del_\beta v$ are considered and certain null conditions are assumed to the nonlinearities, while we can also treat here the nonlinearities $P_1^{\alpha \beta} v \del_\alpha \del_\beta u, P_1^{\alpha \beta} v \del_\alpha \del_\beta v$, and no null conditions are assumed. The small data global existence result to the system \eqref{eq:model-q} and the asymptotic behavior of the solution $(v, w)$ are now stated.

\begin{theorem}\label{thm:main2}
Consider the coupled wave and Klein-Gordon equations in \eqref{eq:model-q}, and let $N \geq 11$ be an integer. There exits $\eps_0 >0$, such that for all $\eps < \eps_0$, and all compactly supported initial data satisfying the smallness condition
\be 
\| v_0 \|_{H^{N+1}} + \| v_1 \|_{H^N}
+
\| w_0 \|_{H^{N+1}} + \| w_1 \|_{H^N}
\leq \eps,
\ee
the Cauchy problem \eqref{eq:model-q} admits a global-in-time solution $(v, w)$, which satisfies the following pointwise decay results 
\bel{eq:thm-decay}
|v(t, x)| \lesssim t^{-1},
\qquad
|\del \del w(t, x) | \lesssim s^{-1}.
\ee
\end{theorem}

We see from Theorem \ref{thm:main2} that the pointwise decay results for $v, \del \del w$ are sharp, which is thanks to the use of a minor different version of the ghost weight energy estimates \eqref{eq:ghost-D}. 

We note that almost the same analysis as the proof of Theorem \ref{thm:main1} also applies to Theorem \ref{thm:main2}, but in order to avoid ambiguities, we also provide the proof for Theorem \ref{thm:main2} in the appendix.



\subsection*{Outline}

The rest of this article is organised as follows.

In Section \ref{sec:BHFM}, we revisit some preliminaries on the wave equations, the hyperboloidal foliation method, and the commutator estimates. 
Then in Section \ref{sec:EE}, we introduce the energy estimates and prepare the Sobolev--type inequalities adapted to the energy estimates. 
Finally we prove Theorem \ref{thm:main1} and Theorem \ref{thm:main2} relying on the bootstrap method in Section \ref{sec:BA}, and in the appendix respectively.


\section{Preliminaries}
\label{sec:BHFM}
 
\subsection{Notations in the hyperboloidal foliation setting}
 
Working in the $(1+2)$ dimensional Minkowski spacetime, we adopt the signature $(-, +, +)$. Recall that the Greek letters $\alpha, \beta, \cdots \in \{0, 1, 2\}$ denote the spacetime indices, and Latin letters $a, b, \cdots \in \{1, 2\}$ represent the space indices, and the indices are raised or lowered by the Minkowski metric $\eta = \text{diag} (-1, 1, 1)$. A point in $\RR^{1+2}$ is denoted by $(x_0, x_1, x_2) = (t, x_1, x_2)$, and we denote its spacial radius by $r = \sqrt{x_1^2 + x_2^2}$. All of the functions considered are assumed to be supported in the cone $\Kcal = \{ (t, x) : t \geq |x| + 1 \}$ (since the solutions are supported in this region). A hyperboloid with hyperbolic time $s$ (with $s \geq 2$) is denoted by $\Hcal_s = \{ (t, x) : t^2 = |x|^2 + s^2 \}$. For a point $(t, x) \in \Hcal_s \bigcap \Kcal$, we emphasize the following important relations
\be 
t \geq |x| + 1,
\qquad
s \leq t \leq s^2.
\ee
Also we use $\Kcal_{[s_0, s_1]} := \{(t, x): s_0^2 \leq t^2- r^2 \leq s_1^2; r\leq t-1 \}$ to denote subsets of $\Kcal$ which are limited by two hyperboloids $\Hcal_{s_0}$ and $\Hcal_{s_1}$ with $s_0 \leq s_1$.

We next introduce some vector fields
\bei
\item Translations: $\del_\alpha$, \quad $\alpha = 0, 1, 2$.

\item Lorentz boosts: $L_a = x_a \del_t + t \del_a$, \quad $a = 1, 2$.

\item Scaling vector field: $L_0 = S = t \del_t + r \del_r$.

\eei
To adapt to the hyperboloidal foliation setting, we introduce the so-callled semi-hyperboloidal frame which is defined by
\bel{eq:semi-hyper}
\underdel_0:= \del_t, 
\qquad 
\underdel_a:= {L_a \over t} = {x^a\over t}\del_t+ \del_a.
\ee
On the other hand, the natural Cartesian frame can be expressed in terms of the semi-hyperboloidal frame as
\be 
\del_t = \underdel_0,
\qquad
\del_a = - {x^a \over t} \del_t + \underdel_a.
\ee


\subsection{Estimates for commutators}

We will need to frequently use the following estimates for commutators, which can be found in \cite{Sogge, PLF-YM-book}

\begin{lemma} \label{lem:comm2}
Let $u$ be a sufficiently nice function supported in $\Kcal = \{ (t, x) : t \geq |x| + 1\}$, then the following inequalities are valid ($a, b, m \in \{1, 2\}, \alpha, \beta \in \{0, 1, 2\}$)
\bel{eq:commu2} 
\aligned
\big|\del_\alpha L_a u \big|
&\lesssim
\big| L_a \del_\alpha u \big| + \sum_\beta \big| \del_\beta u \big|,
\\
\big|\del_\alpha \del_\beta L_a u \big|
&\lesssim
\big| L_a \del_\alpha \del_\beta u \big| + \sum_{\gamma, \gamma'} \big| \del_\gamma \del_{\gamma'} u \big|,
\\
\big| L_a L_b u \big|
&\lesssim
\big| L_b L_a u \big| + \sum_m \big| L_m u \big|,
\\
\big| \del_\alpha (s/t) \big|
&\lesssim
s^{-1},
\\
\big| L_a (s/t) \big|
+ \big| L_b L_a (s/t) \big| 
&\lesssim
s/t,
\\
\big| L_a (t-|x|)^{-\gamma} \big|
&\lesssim
(t-|x|)^{-\gamma},
\\
\big| L_b L_a \big( (s/t)  \big) \big|
&\lesssim (t-|x|)^{-\gamma} {t\over |x|}.
\endaligned
\ee
\end{lemma}
\begin{proof}
We only provide the proof for the last two estimates, as the proof for other inequalities can be found in \cite{Sogge, PLF-YM-book}.

We directly compute
$$
\aligned
L_a (t-|x|)^{-\gamma}
=
&x_a \del_t (t-|x|)^{-\gamma} + t \del_a (t-|x|)^{-\gamma}
\\
=
-&\gamma  (t-|x|)^{-1-\gamma} x_a + \gamma t (t-|x|)^{-1-\gamma} {x_a \over |x|}
\\
=
-&\gamma (t-|x|)^{-1-\gamma} {x_a \over |x|} \big( |x| - t \big)
=
\gamma (t-|x|)^{-\gamma} {x_a \over |x|},
\endaligned
$$ 
which leads to
$$
\big| L_a (t-|x|)^{-\gamma} \big|
\leq \gamma (t-|x|)^{-\gamma}.
$$

Next, we act $L_b$ to the above identity to get ($\delta_{ab}$ is the Kronecker delta function)
$$
\aligned
L_b L_a (t-|x|)^{-\gamma}
=
\gamma^2 (t-|x|)^{-\gamma} {x_a x_b \over |x|^2}
+
\gamma (t-|x|)^{-\gamma} {\delta_{ab} t \over |x|}
-
\gamma (t-|x|)^{-\gamma} {x_a x_b t \over |x|^3}.
\endaligned
$$
Thus the proof is complete.
\end{proof}



\section{Energy estimates on hyperboloids}\label{sec:EE}

\subsection{Energy estimates}

We will introduce two kinds of energy estimates in this subsection: 
in the first kind of energy estimates two ways are shown on how to bound the natural energies for the wave components and the Klein-Gordon components; 
then the second kind of energy estimates allow us to bound the natural energies with some weights, and this is mainly motivated by the ghost weight method.

Let $\phi$ be a sufficiently nice function defined on a hyperboloid $\Hcal_s$, following \cite{PLF-YM-book} we define its natural energy $E_m$ (with three equivalent expressions) by
\bel{eq:2energy} 
\aligned
E_m(s, \phi)
&:=
 \int_{\Hcal_s} \Big( \big(\del_t \phi \big)^2+ \sum_a \big(\del_a \phi \big)^2+ 2 (x^a/t) \del_t \phi \del_a \phi + m^2 \phi ^2 \Big) \, dx
\\
               &= \int_{\Hcal_s} \Big( \big( (s/t)\del_t \phi \big)^2+ \sum_a \big(\underdel_a \phi \big)^2+ m^2 \phi^2 \Big) \, dx
                \\
               &= \int_{\Hcal_s} \Big( \big( \underdel_\perp \phi \big)^2+ \sum_a \big( (s/t)\del_a \phi \big)^2+ \sum_{a<b} \big( t^{-1}\Omega_{ab} \phi \big)^2+ m^2 \phi^2 \Big) \, dx,
 \endaligned
 \ee
in which $\Omega_{ab} := x^a\del_b- x^b\del_a$ are the rotation vector fields, and $\underdel_{\perp} := L_0/t = \del_t+ (x^a / t) \del_a$ is the orthogonal vector field. 
The above integral is defined by
\bel{eq:flat-int}
\int_{\Hcal_s}|\phi | \, dx 
=\int_{\RR^2} \big|\phi(\sqrt{s^2+r^2}, x) \big| \, dx,
\ee
and we denote
\be 
\| \phi \|_{L^p_f(\Hcal_s)}
=
\Big( \int_{\Hcal_s} |\phi|^p \, dx \Big)^{1/p},
\qquad
1\leq p < +\infty,
\ee
while $\| \phi \|_{L_f^\infty(\Hcal_s)} := \| \phi \|_{L^\infty(\Hcal_s)}$.
Note that the second and the third expressions in \eqref{eq:2energy} yield
$$
\big\| (s/t) \del \phi \big\|_{L^2_f(\Hcal_s)} + \sum_a \big\| \underdel_a \phi \big\|_{L^2_f(\Hcal_s)}
\lesssim
E_m(s, \phi)^{1/2}.
$$

\paragraph{Energy estimates I}

Now, we demonstrate the first energy estimates to the hyperboloidal setting.

\begin{proposition}[Energy estimates for wave-Klein-Gordon equations]
For $m \geq 0$ and for $s \geq s_0$ (with $s_0 = 2$), it holds both
\bel{eq:w-EE} 
E_m(s, \phi)^{1/2}
\leq 
E_m(s_0, \phi)^{1/2}
+ \int_2^s \big\| -\Box \phi + m^2 \phi \big\|_{L^2_f(\Hcal_{s'})} \, ds'
\ee
and 
\bel{eq:w-EE2} 
E_m(s, \phi)
\leq 
E_m(s_0, \phi)
+ \int_2^s \int_{\Hcal_{s'}} (s'/t) \big| \del_t \phi\big| \big| -\Box \phi + m^2 \phi \big| \, dxds'
\ee
for all sufficiently regular functions $\phi = \phi(t, x)$, which are defined and supported in $\Kcal_{[s_0, s]}$.
\end{proposition}

One refers to \cite{PLF-YM-cmp} for the proof. A side remark is that the energy estimates \eqref{eq:w-EE2} allows us to show the uniform energy bounds for the $E$ component, while the energy estimates \eqref{eq:w-EE} cannot achieve this, and the details can be found in the proof of Proposition \ref{prop:refine}.

\paragraph{Energy estimates II}

To proceed, we introduce a minor different version of the ghost weight energy estimates for the wave equations, which can help compensate the loss of $t$ decay by the (less important in many cases) loss of $t-r$ decay. Roughly speaking, the version of ghost weight energy estimates below allows us to show 
$$
| \del n^\Delta | \lesssim s^{-1} (t-r)^{2\delta}
$$
from
$$
| \del n^\Delta | \lesssim s^{-1 + \delta}.
$$

Consider the wave equation
\bel{eq:u-equation}
-\Box u = f,
\ee
and recall that in the original ghost weight method by Alinhac, the multiplier used is 
$$
e^{\arctan (r-t)} \del_t u,
$$
which does not contribute to the right hand side compared with the usual multiplier $\del_t u$, and this is because $e^{\arctan (r-t)}$ is only as good as a constant. However, we find that if we instead use 
$$
(t-r)^{-\gamma} \del_t u
$$
as the multiplier, the right hand side can benefit from the $(t-r)^{-\gamma}$ factor. And thanks to the contribution of the factor $(t-r)^{-\gamma}$, the original non-integrable quantity might turn to the integrable quantity $s^{-1-\gamma'}$ with $\gamma' >0$. Note that we are allowed to benefit from the $t-r$ decay because the functions considered are supported in the region $t-r \geq 1$.

\begin{proposition}\label{prop:ghost}
Consider the wave equation \eqref{eq:u-equation} and assume $u$ is supported in $\Kcal = \{(t, x): |x|\leq t - 1\}$, then we have the following version of ghost weight energy estimates
\bel{eq:ghost-D}
\int_{\Hcal_s} (t-r)^{-\gamma} \big(s/t \big)^2 \big| \del u \big|^2 \, dx
\leq
2 E(s_0, u)
+
4 \int_{s_0}^s \int_{\Hcal_{s'}} (s'/t) (t-r)^{-\gamma} f \del_t u \, dxds',
\ee
in which $\gamma > 0$.
\end{proposition}

\begin{proof}
Multiplying on both sides of \eqref{eq:u-equation} with $(t-r)^{-\gamma} \del_t u$, we get
$$
\aligned
&-\Box u \cdot (t-r)^{-\gamma} \del_t u
=
f \cdot (t-r)^{-\gamma} \del_t u 
\\
=
&{1\over 2} \del_t \Big( (t-r)^{-\gamma} \big( (\del_t u)^2 + \sum_a (\del_a u)^2 \big) \Big)
- \del_a \Big( (t-r)^{-\gamma} \del^a u \del_t u \Big)
\\
+
&{\gamma \over 2} (t-r)^{-1-\gamma} \sum_a \big( (x_a /r) \del_t u + \del_a u \big)^2,
\endaligned
$$
and the observation that
$$
{\gamma \over 2} (t-r)^{-1-\gamma} \sum_a \big( (x_a /r) \del_t u + \del_a u \big)^2
\geq 0
$$
further gives
$$
{1\over 2} \del_t \Big( (t-r)^{-\gamma} \big( (\del_t u)^2 + \sum_a (\del_a u)^2 \big) \Big)
- \del_a \Big( (t-r)^{-\gamma} \del^a u \del_t u \Big)
\leq 
f \cdot (t-r)^{-\gamma} \del_t u.
$$
Then we integrate it over the region $\Kcal_{[s_0, s]}$ to get
$$
\aligned
&\int_{\Kcal_{[s_0, s]}} {1\over 2} \del_t \Big( (t-r)^{-\gamma} \big( (\del_t u)^2 + \sum_a (\del_a u)^2 \big) \Big)
- \del_a \Big( (t-r)^{-\gamma} \del^a u \del_t u \Big) \, dxdt
\\
\leq 
&\int_{\Kcal_{[s_0, s]}} f \cdot (t-r)^{-\gamma} \del_t u \, dxdt.
\endaligned
$$
Note the out unit normal to the hyperboloid $\Hcal_s$ is $(t, -x^a) \cdot (t^2 + |x|^2)^{-1/2}$ and $d\Hcal_s = t^{-1/2} (t^2 + |x|^2)^{1/2} dx$, and the Stokes formula yields
$$
\aligned
&\int_{\Hcal_s} (t-r)^{-\gamma} \Big( \big(\del_t u \big)^2+ \sum_a \big(\del_a u \big)^2+ 2 (x^a/t) \del_t u \del_a u \Big) \, dx
\\
\leq
&\int_{\Hcal_{s_0}} (t-r)^{-\gamma} \Big( \big(\del_t u \big)^2+ \sum_a \big(\del_a u \big)^2+ 2 (x^a/t) \del_t u \del_a u \Big) \, dx
+
2 \int_{s_0}^s \int_{\Hcal_{s'}} (s'/t) f \cdot (t-r)^{-\gamma} \del_t u \, dxds',
\endaligned
$$
in which we also used the relation $dxdt = (s/t) dxds$.

Finally, recalling those three equivalent expressions for the energy $E (s, u)$ in \eqref{eq:2energy} finishes the proof.
\end{proof}

Note that the energy estimates in \eqref{eq:ghost-D} exclude the contribution of the positive spacetime integral of the derivatives $(x_a/|x|) \del_t u + \del_a u$, which is heavily relied on in the original ghost weight method. The reason why we do not include that contribution in the energy estimates is that: 1) in the hyperboloidal foliation setting, the derivatives $(x_a/|x|) \del_t + \del_a$ do not seem to be so good (do not seem to be as good as $\underdel_a$); 2) in the models of interest \eqref{eq:model-KGZ} and \eqref{eq:model-q}, there do not exist any null structures.

We will see later that this proposition plays a vital role in showing 
$$
| \del \del n^\Delta | \lesssim s^{-1},
\qquad
(\text{and }  | \del \del w| \lesssim s^{-1},)
$$
for the equations \eqref{eq:model-KGZ2} (and \eqref{eq:model-q}), and this is necessary in order to close the bootstrap argument.


\subsection{Sobolev-type inequalities}

We now state a Sobolev-type inequality adapted to the hyperboloidal foliation setting, which will be used to obtain pointwise estimates for both wave and Klein-Gordon components. The Sobolev-type inequalities on hyperboloids have been proved by Klainerman, H{\"o}rmander, and LeFloch-Ma. For the proof of the one given right below, one refers to either \cite{Sogge} or \cite{PLF-YM-book, PLF-YM-cmp} for details.

\begin{lemma} \label{lem:sobolev}
Let $u= u(t, x)$ be sufficient smooth and be supported in $\{(t, x): |x|\leq t - 1\}$ and let $s \geq 2$, then it holds 
\bel{eq:Sobolev1}
\sup_{\Hcal_s} \big| t u(t, x) \big|  
\lesssim \sum_{| J |\leq 2} \big\| L^J u \big\|_{L^2_f(\Hcal_s)},
\ee
with $L$ the Lorentz boosts and $J$ the multi-index. 
\end{lemma}

Combine with the estimates for commutators in Lemma \ref{lem:comm2}, we also have the following more practical versions of Sobolev inequalities.
\begin{lemma}\label{lem:sobolev2}
Under the same assumptions as in Lemma \ref{lem:sobolev}, we have
\bel{eq:Sobolev2}
\sup_{\Hcal_s} \big| s \hskip0.03cm u(t, x) \big| 
\lesssim \sum_{| J |\leq 2} \big\| (s/t) L^J u \big\|_{L^2_f(\Hcal_s)},
\ee
and 
\bel{eq:Sobolev3}
\sup_{\Hcal_s} \big| s (t-r)^{-\gamma} \hskip0.03cm u(t, x) \big| 
\lesssim \sum_{| J |\leq 2} \big\| (s/t) (t-r)^{-\gamma} L^J u \big\|_{L^2_f(\Hcal_s)},
\ee
\end{lemma}
\begin{proof}
To prove \eqref{eq:Sobolev2}, we first note
$$
\sup_{\Hcal_s} \big| s \hskip0.03cm u(t, x) \big| 
=
\sup_{\Hcal_s} \big| t \hskip0.03cm (s/t) u(t, x) \big|,
$$
then \eqref{eq:Sobolev2} follows from the commutator estimates in \eqref{eq:commu2}
$$
\big| L_\alpha (s/t) \big|
+ \big| L_b L_a (s/t) \big| 
\lesssim
s/t.
$$

Next, we show \eqref{eq:Sobolev3}, and due to the fact that $(t-r)^{-\gamma}$ is not sufficiently smooth, its proof is a little bit more complicated.
We introduce the smooth cut-off function
\begin{eqnarray}
\phi_0 (p)
=
\left\{
\begin{array}{lll}
&
0,
\qquad
&  p < {\sqrt{3}\over 2},
\\
&
1,
\qquad
& p > {2\sqrt{2} \over 3}.
\end{array}
\right.
\end{eqnarray}
Let
$$
\phi (t, x)
=
\phi_0 (s/t),
$$
and we observe that
\begin{eqnarray}
\phi (t, x)
=
\left\{
\begin{array}{lll}
&
0,
\qquad
&  t <  2 |x|,
\\
&
1,
\qquad
& t > 3 |x|.
\end{array}
\right.
\end{eqnarray}
In addition, we also find
$$
\sum_{|J| \leq 2} \big| L^J \phi(t, x) \big|
\lesssim 1.
$$

The simple triangle inequality gives
$$
\sup_{\Hcal_s} \big| s (t-r)^{-\gamma} \hskip0.03cm u(t, x) \big| 
\leq 
\sup_{\Hcal_s} \big| \phi(t, x) s (t-r)^{-\gamma} \hskip0.03cm u(t, x) \big| 
+
\sup_{\Hcal_s} \big| \big(1- \phi(t,x)\big) s (t-r)^{-\gamma} \hskip0.03cm u(t, x) \big|, 
$$
which means it suffices to show 
\bel{eq:suffice}
\aligned
\sup_{\Hcal_s} \big| \phi(t, x) s t^{-\gamma} \hskip0.03cm u(t, x) \big| 
&\lesssim \sum_{| J |\leq 2} \big\| (s/t) t^{-\gamma} L^J u \big\|_{L^2_f(\Hcal_s)},
\\
\sup_{\Hcal_s} \big| \big(1- \phi(t,x)\big) s (t-r)^{-\gamma} \hskip0.03cm u(t, x) \big|
&\lesssim \sum_{| J |\leq 2} \big\| (s/t) (t-r)^{-\gamma} L^J u \big\|_{L^2_f(\Hcal_s)},
\endaligned
\ee
in which we used the relation that $t \lesssim t-r \lesssim t$ within the support of $\phi(t, x)$.
The second estimate can be seen from the commutator estimates in \eqref{eq:commu2}
$$
\big| L_a (t-|x|)^{-\gamma} \big|
\lesssim
(t-|x|)^{-\gamma},
\qquad
\big| L_b L_a \big( (s/t)  \big) \big|
\lesssim (t-|x|)^{-\gamma} {t\over |x|},
$$
as well as the fact that
$$
t \lesssim |x|
$$
holds within the support of $1-\phi(t, x)$ which is $\{(t, x) : |x| \geq t/3 \}$.
In order to show the first estimate in \eqref{eq:suffice}, we compute
$$
L_a t^{-\gamma}
=
-\gamma t^{-\gamma} {x_a \over t},
\qquad
L_b L_a t^{-\gamma}
=
\gamma (1+\gamma) t^{-\gamma} {x_a x_b \over t^2} 
- 
\gamma t^{-\gamma} \delta_{ab},
$$
and these imply the first estimate in \eqref{eq:suffice}.

The proof is complete.
\end{proof}


\section{Bootstrap argument}
\label{sec:BA}

\subsection{Bootstrap assumptions and its consequences}\label{subsec:BA}

We will work on the equations in \eqref{eq:model-KGZ2}, and for easy readability we copy \eqref{eq:model-KGZ2}
$$
\aligned
&-\Box E^a + E^a = \Delta n^\Delta E^a,
\qquad
&\big( E^a, \del_t E^a \big)(t_0)
= (E^a_0, E^a_1),
\\
&-\Box n^\Delta = (E^1)^2 + (E^2)^2,
\qquad
&\big( n^\Delta, \del_t n^\Delta  \big) (t_0)
= (n^\Delta_0, n^\Delta_0).
\endaligned
$$

We assume that it holds for $s \in [s_0, s_1)$
\bel{eq:BA-KGZ}
\aligned
E_1 (s, \del^I L^J E)^{1/2}
+
E_1 (s, \del \del^I L^J E)^{1/2}
&\leq C_1 \eps s^\delta,
\qquad
&|I| + |J| \leq N,
\\
E_1 (s, \del^I L^J E)^{1/2}
&\lesssim C_1 \eps,
\qquad
&|I| + |J| \leq N-3.
\endaligned
\ee
In the above, $0 <\delta< 1/24$, $C_1 > 1$ is a large number to be determined satisfying $C_1 \eps \ll \delta$, and $s_1$ is defined by
\be 
s_1 = \sup \{ s : s > s_0, \,\eqref{eq:BA-KGZ} \,\, holds \}.
\ee
In what follows, we first assume $s_1 > s_0$ is a finite number, and we then derive some contradiction to assert that $s_1 = + \infty$, so that we have the global existence result.

As a consequence of the bootstrap assumptions, we have the following estimates.

\begin{lemma}
Assume the bounds in \eqref{eq:BA-KGZ} are valid, then for all $s \in [s_0, s_1)$ it holds the following $L^2$ norm estimates
\bel{eq:BA-E-L2}
\aligned
\big\| \del^I L^J E, \del \del^I L^J E \big\|_{L^2_f(\Hcal_s)} 
+
\big\| (s/t) \del \del \del^I L^J E \big\|_{L^2_f(\Hcal_s)}
&\lesssim C_1 \eps s^\delta,
\quad
&|I| + |J| \leq N,
\\
\big\|\del^I L^J E \big\|_{L^2_f(\Hcal_s)} + \big\| (s/t) \del \del^I L^J E \big\|_{L^2_f(\Hcal_s)}
&\lesssim C_1 \eps,
\quad
&|I| + |J| \leq N-3,
\endaligned
\ee
as well as the following pointwise estimates
\bel{eq:BA-E-L00} 
\aligned
|\del^I L^J E |
&\lesssim C_1 \eps t^{-1} s^\delta,
\qquad
&|I| + |J| \leq N-2,
\\
\big|\del^I L^J E \big|
&\lesssim C_1 \eps t^{-1},
\qquad
&|I| + |J| \leq N-5.
\endaligned
\ee
\end{lemma}
\begin{proof}
We note that the $L^2$--type estimates in \eqref{eq:BA-E-L2} follow directly from the definition of the energy $E_1 (s, E)$, while the pointwise estimates in \eqref{eq:BA-E-L00} can be obtained by using the Sobolev--type inequality \eqref{eq:Sobolev1}.
\end{proof}

\subsection{Refined estimates} \label{subsec:refine}

We first derive the energy estimates for the $n^\Delta$ variable.

\begin{lemma}
Let the bounds in \eqref{eq:BA-KGZ} hold, then we have
\bel{eq:n12-1} 
\aligned
&E (s, \del^I L^J n^\Delta)^{1/2}
+
E (s, \del \del^I L^J n^\Delta)^{1/2}
+
E (s, \del \del \del^I L^J n^\Delta)^{1/2}
\\
\lesssim
&\eps + (C_1 \eps)^2 s^\delta,
\qquad \qquad
|I| + |J| \leq N,
\endaligned
\ee
as well as
\bel{eq:n12-2}
\aligned
&\Big\|(t-|x|)^{-2\delta} ( s/t ) \del \del^I L^J n^\Delta  \Big\|_{L^2_f(\Hcal_s)}
\lesssim \eps + (C_1 \eps)^{3/2},
\qquad
&|I| + |J| \leq N,
\\
&\big| \del \del^I L^J n^\Delta \big|
\lesssim \big( \eps + (C_1 \eps)^{3/2} \big) s^{-1} (t-|x|)^{2\delta},
\qquad
&|I| + |J| \leq N-2.
\endaligned
\ee
\end{lemma}
\begin{proof}

We act $\del \del \del^I L^J$ with $|I| + |J| \leq N$ on the $n^\Delta$ equation to get
$$
-\Box \del \del \del^I L^J n^\Delta
=
\del \del \del^I L^J \big( (E^1)^2 + (E^2)^2 \big).
$$

Consider first the estimates in \eqref{eq:n12-1}, and we note that (recall $N\geq 11$)
$$
\aligned
&\big\| \del \del \del^I L^J \big( (E^1)^2 + (E^2)^2 \big) \big\|_{L^2_f(\Hcal_s)}
\\
\lesssim
&\sum_{\substack{|I_1|+|J_1| \leq N\\ |I_2| + |J_2| \leq N-5} }
\big\| (s/t) \del \del \del^{I_1} L^{J_1} E \big\|_{L^2_f(\Hcal_s)} \big\| (t/s) \del^{I_2} L^{J_2} E \big\|_{L^\infty(\Hcal_s)} 
\\
+ &\sum_{\substack{|I_1|+|J_1| \leq N\\ |I_2| + |J_2| \leq N-5} } \big\| \del^{I_1} L^{J_1} E \big\|_{L^2_f(\Hcal_s)} \big\| \del^{I_2} L^{J_2} E \big\|_{L^\infty(\Hcal_s)}  
\\
\lesssim
&(C_1 \eps)^2 s^{-1+\delta}.
\endaligned
$$
Then the energy estimates \eqref{eq:w-EE} implies that
$$
\aligned
E (s, \del \del \del^I L^J n^\Delta)^{1/2}
\lesssim
& E (s_0, \del \del \del^I L^J n^\Delta)^{1/2}
+
\int_{s_0}^s \big\| \del \del \del^I L^J \big( (E^1)^2 + (E^2)^2 \big) \big\|_{L^2_f(\Hcal_{s'})} \, ds'
\\
\lesssim
& \eps + (C_1 \eps)^2 \int_{s_0}^s s'^{-1+\delta} \, ds'
\lesssim
\eps + (C_1 \eps)^2 s^\delta.
\endaligned
$$
In the same way,  we obtain
$$
E (s, \del^I L^J n^\Delta)^{1/2} + E (s, \del \del^I L^J n^\Delta)^{1/2}
\lesssim
\eps + (C_1 \eps)^2 s^\delta.
$$

Next, we turn to prove the estimates in \eqref{eq:n12-2}. We apply the ghost weight energy estimates \eqref{eq:ghost-D} on the equation
$$
-\Box \del^I L^J n^\Delta
=
\del^I L^J \big( (E^1)^2 + (E^2)^2 \big),
$$
to arrive at
$$
\aligned
&\int_{\Hcal_s} (t-r)^{-4\delta} \big(s/t \big)^2 \big| \del \del^I L^J n^\Delta \big|^2 \, dx
\\
\lesssim
& E(s_0, n^\Delta)
+
\int_{s_0}^s \int_{\Hcal_{s'}} (s'/t) (t-r)^{-4\delta} \del_t \del^I L^J n^\Delta \del^I L^J \big( (E^1)^2 + (E^2)^2 \big) \, dxds'
\\
\lesssim
& \eps^2 + \int_{s_0}^s \big\| (s'/t) \del_t \del^I L^J n^\Delta \big\|_{L^2_f(\Hcal_{s'})} \big\| (t-r)^{-4\delta} \del^I L^J \big( (E^1)^2 + (E^2)^2 \big) \big\|_{L^2_f(\Hcal_{s'})} \, ds'.
\endaligned
$$
Observe that
$$
\aligned
&\big\| (t-r)^{-4\delta} \del^I L^J \big( (E^1)^2 + (E^2)^2 \big) \big\|_{L^2_f(\Hcal_s)}
\\\lesssim
&\sum_{\substack{|I_1|+|J_1| \leq N\\ |I_2| + |J_2| \leq N-2} }
\big\| \del^{I_1} L^{J_1} E \big\|_{L^2_f(\Hcal_s)} \big\| (t-r)^{-4\delta} \del^{I_2} L^{J_2} E \big\|_{L^\infty(\Hcal_s)} 
\\
\lesssim
& (C_1 \eps)^3 s^{-1-2\delta},
\endaligned
$$
and, we thus obtain
$$
\aligned
\int_{\Hcal_s} (t-r)^{-4\delta} \big(s/t \big)^2 \big| \del \del^I L^J n^\Delta \big|^2 \, dx
\lesssim \eps^2 + (C_1 \eps)^3 \int_{s_0}^s s'^{-1-\delta} \, ds'
\lesssim \eps^2 + (C_1 \eps)^3.
\endaligned
$$
Finally, we apply the Sobolev--type inequality \eqref{eq:Sobolev3} to deduce the pointwise estimates appearing in \eqref{eq:n12-2}.
\end{proof}

We need the following result on the estimates of wave components with second order partial derivatives, which was used in \cite{MaH1}.

\begin{lemma}
Let $u = u(t, x)$ be a sufficiently nice function with support $\Kcal$, then it hold
\bel{eq:ddu}
\big| \del \del u \big|
\lesssim
(t-|x|)^{-1} \big( \big| \del L u \big| + \big| \del u\big|  \big) + {t \over t-|x|} \big|\Box u \big|.
\ee
As a consequence, we have
\bel{eq:n12-5}
\aligned
&\big| \del \del \del^I L^J n^\Delta \big|
\lesssim \big(\eps + (C_1 \eps)^{3/2} \big) s^{-1} (t-|x|)^{-1/2},
\qquad
&|I| + |J| \leq N-3,
\\
&\big\| (s/t) (t-|x|)^{1/2} \del \del \del^I L^J n^\Delta \big\|_{L^2_f(\Hcal_s)}
\lesssim \eps + (C_1 \eps)^{3/2},
\qquad
&|I| + |J| \leq N-1,
\endaligned
\ee
\end{lemma}
\begin{proof}
The proof of \eqref{eq:ddu} can be found in \cite{MaH1, Ma17, Dong2005}.
And, we will only show 
$$
\big| \del \del \del^I L^J n^\Delta \big|
\lesssim \big(\eps + (C_1 \eps)^{3/2} \big) s^{-1} (t-|x|)^{-1/2},
\qquad
|I| + |J| \leq N-3,
$$
as the other estimate in \eqref{eq:n12-5} can be derived in a very similar way.

According to \eqref{eq:ddu}, we have
$$
\aligned
\big| \del \del \del^I L^J n^\Delta \big|
\lesssim
(t-|x|)^{-1} \big( \big| \del L \del^I L^J n^\Delta \big| + \big| \del \del^I L^J n^\Delta\big|  \big) + {t \over t-|x|} \big|\Box \del^I L^J n^\Delta \big|.
\endaligned
$$
On one hand, the estimates for commutators give
$$
\big| \del L \del^I L^J n^\Delta \big|
\lesssim
\sum_{|I_1| \leq |I|} \big( \big| \del \del^{I_1} L L^J n^\Delta \big| + \big| \del \del^{I_1} L^J n^\Delta \big| \big),
$$
on the other hand, the equation of $n^\Delta$ in \eqref{eq:model-KGZ2} implies
$$
\big|\Box \del^I L^J n^\Delta \big|
\leq \big| \del^I L^J |E|^2 \big|
\lesssim
\sum_{\substack{|I_1| + |I_2| \leq |I|\\ |J_1| + |J_2| \leq |J|}}  \big| \del^{I_1} L^{J_1} E \big|   \big| \del^{I_2} L^{J_2} E \big|.
$$
Successively, we have
$$
\big| \del \del \del^I L^J n^\Delta \big|
\lesssim
\big( \eps + (C_1 \eps)^{3/2} \big)  (t-|x|)^{-1} s^{-1} (t-r)^{2\delta} + (C_1 \eps)^2 {t \over t-|x|} t^{-2+2\delta}.
$$
Finally the smallness of $\delta$ yields the desired estimates, and the proof is done.
\end{proof}

We are now ready to provide the refined estimates of $E$ component.

\begin{proposition}\label{prop:refine}
Under the bootstrap assumptions in \eqref{eq:BA-KGZ}, the following estimates hold
\bel{eq:refine-2}
\aligned
E_1 (s, \del^I L^J E)^{1/2}
+
E_1 (s, \del \del^I L^J E)^{1/2}
&\lesssim \eps + (C_1 \eps)^2 s^\delta,
\qquad
&|I| + |J| \leq N,
\\
E_1 (s, \del^I L^J E)^{1/2}
&\lesssim \eps + (C_1 \eps)^{3/2},
\qquad
&|I| + |J| \leq N-3.
\endaligned
\ee
\end{proposition}
\begin{proof}
Act the vector filed $\del^I L^J$ on the $E^a$ equation, and we get
$$
-\Box \del^I L^J E^a + \del^I L^J E^a = \del^I L^J \big(\Delta n^\Delta E^a \big).
$$

We first provide the proof for the high order energy cases, with $|I| + |J| \leq N$.
The energy estimates in \eqref{eq:w-EE} imply that
$$
E_1 (s, \del^I L^J E)^{1/2}
\leq
E_1 (s_0, \del^I L^J E)^{1/2}
+
\int_{s_0}^s \big\| \del^I L^J \big(\Delta n^\Delta E \big) \big\|_{L^2_f(\Hcal_{s'})} \, ds'.
$$
Easily we find that
$$
\aligned
\big\| \del^I L^J \big(\Delta n^\Delta E^a \big) \big\|_{L^2_f(\Hcal_s)}
\lesssim
&\sum_{\substack{|I_1| + |J_1| \leq N\\ |I_2| + |J_2| \leq N-5} }
\Big( \big\| (s/t) \del \del \del^{I_1} L^{J_1} n^\Delta \big\|_{L^2_f(\Hcal_s)} \big\| (t/s) \del^{I_2} L^{J_2} E \big\|_{L^\infty(\Hcal_s)} 
\\
& \hskip2cm + \big\| \del \del \del^{I_2} L^{J_2} n^\Delta \big\|_{L^\infty(\Hcal_s)} \big\| \del^{I_1} L^{J_1} E \big\|_{L^2_f(\Hcal_s)}   \Big)
\\
\lesssim 
& (C_1 \eps)^2 s^{-1+\delta},
\endaligned
$$
which further deduces that
$$
E_1 (s, \del^I L^J E)^{1/2}
\lesssim \eps + (C_1 \eps)^2 s^\delta,
\qquad
|I| + |J| \leq N.
$$
In the same way, we also obtain
$$
E_1 (s, \del \del^I L^J E)^{1/2}
\lesssim \eps + (C_1 \eps)^2 s^\delta,
\qquad
|I| + |J| \leq N.
$$

Note , however, that the energy estimates in \eqref{eq:w-EE} cannot be used to show the uniform energy estimates of $E$,
so we rely on the trick here that we turn to the energy estimates in \eqref{eq:w-EE2} which read as follows
$$
E_1 (s, \del^I L^J E)
\leq
E_1 (s_0, \del^I L^J E)
+
\int_{s_0}^s \int_{\Hcal_{s'}} (s'/t) \big| \del_t \del^I L^J E \big|  \big| \del^I L^J \big(\Delta n^\Delta E \big) \big| \, dxds'.
$$
The important thing is that we can move the good factor $s'/t$ to the function of $n^\Delta$, which helps circumvent the lack of $t$ decay of $n^\Delta$ part.
For $|I| + |J| \leq N-3$, we have
$$
\aligned
&\int_{\Hcal_s} (s/t) \big| \del_t \del^I L^J E \big|  \big| \del^I L^J \big(\Delta n^\Delta E \big) \big| \, dx
\\
\lesssim
&\sum_{|I_1| + |I_2| + |J_1| + |J_2| \leq N-3} \big\| (s/t) (t-|x|)^{1/2} \del \del \del^{I_1} L^{J_1} n^\Delta \big\|_{L^2_f(\Hcal_s)}
 \big\| \del_t  \del^I L^J  E\big\|_{L^2_f(\Hcal_s)}
\\
&\hskip3.2cm \cdot 
\big\| (t-|x|)^{-1/2} \del^{I_2} L^{J_2} E \big\|_{L^\infty(\Hcal_s)}
\\
\lesssim
& (C_1 \eps)^3 s^{-3/2 + 3 \delta},
\endaligned
$$
which is integrable as $\delta$ is small.
Hence, we further have
$$
E_1 (s, \del^I L^J E)^{1/2}
\lesssim
\eps + (C_1 \eps)^{3/2}
$$
as desired.

We thus complete the proof.
\end{proof}

The proof of Theorem \ref{thm:main1} follows.
\begin{proof}[Proof of Theorem \ref{thm:main1}]
By choosing $C_1$ large enough, and $\eps$ sufficiently small, such that $C_1 \eps \ll \delta$, the estimates in \eqref{eq:refine-2} imply that
$$
\aligned
E_1 (s, \del^I L^J E)^{1/2}
+
E_1 (s, \del \del^I L^J E)^{1/2}
&\leq {1\over 2} C_1 \eps s^\delta,
\qquad
&|I| + |J| \leq N,
\\
E_1 (s, \del^I L^J E)^{1/2}
&\leq {1\over 2} C_1 \eps,
\qquad
&|I| + |J| \leq N-3
\endaligned
$$
are valid for all $s\in [s_0, s_1)$. This means $s_1$ must be $+\infty$, and thus the Klein-Gordon-Zakharov equations \eqref{eq:model-KGZ} admit the global solution $(E, n)$. 

The sharp pointwise decay of $E$ is from \eqref{eq:BA-E-L00}, while the sharp pointwise decay of $n = \Delta n^\Delta$ can be seen from \eqref{eq:n12-5}. On the other hand, the energy estimates \eqref{eq:thm-E} can be obtained from \eqref{eq:BA-KGZ} and \eqref{eq:n12-1}.
\end{proof}


\section*{Appendix: Proof of Theorem \ref{thm:main2}}\label{sec:Appendix}

\subsection*{Energy estimates for quasilinear wave and Klein-Gordon equations}

We first rewrite the equations in \eqref{eq:model-q} in the following form
\bel{eq:model-q-A}
\aligned
&-\Box v + v + Q^{\alpha \beta} \del_\alpha \del_\beta w = 0,
\\
&-\Box w + Q^{\alpha \beta} \del_\alpha \del_\beta v = 0,
\\
&\big( v, \del_t v, w, \del_t w \big)(t_0) = (v_0, v_1, w_0, w_1),
\endaligned
\ee 
in which
\be 
Q^{\alpha \beta}
:= P_1^{\alpha \beta} v + P_2^{\alpha \beta \gamma} \del_\gamma v.
\ee
Without loss of generality, we assume the following symmetry conditions
\be 
P_1^{\alpha \beta} = P_1^{ \beta \alpha},
\qquad
P_1^{\alpha \beta \gamma} = P_1^{ \beta \alpha \gamma},
\ee
which imply that
\be 
Q^{\alpha \beta} = Q^{\beta \alpha}.
\ee

We define the energy for the quasilinear system \eqref{eq:model-q-A} as
\bel{eq:E-q} 
\aligned
E (s, v, w)
=
&E_1 (s, v) + E (s, w) 
+
\int_{\Hcal_s} \Big( Q^{0\beta} \del_\beta v \del_t w + Q^{0\beta} \del_\beta w \del_t v - Q^{\alpha \beta} \del_\alpha v \del_\beta w
\\
& \hskip4cm - (x_a /t) \big( Q^{a\beta} \del_\beta v \del_t w + Q^{a\beta} \del_\beta w \del_t v \big) \Big) \, dx.
\endaligned
\ee
Now the energy estimates for the quasilinear system \eqref{eq:model-q-A} is illustrated.

\begin{proposition}\label{prop:EE-q}
Consider the system
$$
\aligned
&-\Box v + v + Q^{\alpha \beta} \del_\alpha \del_\beta w = f,
\\
&-\Box w + Q^{\alpha \beta} \del_\alpha \del_\beta v = g,
\\
&\big( v, \del_t v, w, \del_t w \big)(t_0) = (v_0, v_1, w_0, w_1),
\endaligned
$$
then we have 
\bel{eq:EE-q-A1}
\aligned
E (s, v, w)
\leq 
&E (s_0, v, w)
+
\int_{s_0}^s \int_{\Hcal_{s'}} (s'/t) \big( \del_\alpha Q^{\alpha \beta} \del_\beta v \del_t w + \del_\alpha Q^{\alpha\beta} \del_\beta w \del_t v 
\\
&\hskip4.3cm - \del_t Q^{\alpha\beta} \del_\beta v \del_\alpha w + f \del_t v + g \del_t w  \big) \, dxds'.
\endaligned
\ee

\end{proposition}
\begin{proof}
The proof is standard, and we note that the following identity can be shown using the symmetry property of $Q^{\alpha\beta}$
$$
\aligned
&\big(-\Box v + v + Q^{\alpha \beta} \del_\alpha \del_\beta w  \big) \del_t v
+
\big(-\Box w + Q^{\alpha \beta} \del_\alpha \del_\beta v  \big) \del_t w
=
f \del_t v + g \del_t w
\\
=
&{1\over 2} \del_t \big( (\del_t v)^2 + \sum_a (\del_a v)^2 + v^2 + (\del_t w)^2 + \sum_a (\del_a w)^2 \big)
-
\del_a \big( \del^a v \del_t v \big) 
-
\del_a \big( \del^a w \del_t w \big) 
\\
+
&\del_\alpha \big( Q^{\alpha\beta} \del_\beta w \del_t v \big)
+
\del_\alpha \big( Q^{\alpha\beta} \del_\beta v \del_t w \big) 
-
\del_t \big( Q^{\alpha\beta} \del_\beta v \del_\alpha w \big)
\\
-
&\del_\alpha Q^{\alpha \beta} \del_\beta w \del_t v
-
\del_\alpha Q^{\alpha\beta} \del_\beta v \del_t w
+
\del_t Q^{\alpha\beta} \del_\beta v \del_\alpha w.
\endaligned
$$
Then a similar computation to the proof of Proposition \ref{prop:ghost} leads us to the desired results in \eqref{eq:EE-q-A1}.
\end{proof}

We also have the following result.

\begin{lemma}\label{lem:equav}
Consider the energy $E(s, v, w)$ defined in \eqref{eq:E-q}, and we assume that
\be 
\big| Q^{\alpha\beta} \big|
\ll {1\over 100} {s^2 \over t^2}.
\ee
Then it holds that
\bel{eq:EE-e}
E (s, v, w)
\lesssim
E_1 (s,v ) + E (s, w)
\lesssim
E (s, v, w).
\ee
\end{lemma}

\subsection*{Bootstrap assumptions}

As usual, we assume the following bootstrap assumptions hold for $s \in [s_0, s_1)$
\bel{eq:BA-A}
\aligned
E_1 (s, \del^I L^J v)^{1/2} + E (s, \del^I L^J w)^{1/2}
&\leq C_1 \eps s^\delta,
\qquad
&|I| + |J| \leq N,
\\
E_1 (s, \del^I L^J v)^{1/2}
&\leq C_1 \eps,
\qquad
&|I| + |J| \leq N-3,
\\
\big| \del \del \del^I L^J w (t, x) \big|
&\leq C_1 \eps s^{-1} (t-r)^{-1/2},
\qquad
&|I| + |J| \leq N-5,
\endaligned
\ee
in which $C_1 > 1$ is some big number to be determined, and satisfies $C_1 \eps \ll 1/100$, and $s_1 > s_0$ is defined by
\be 
s_1 = \sup \{ s : s > s_0, \,\eqref{eq:BA-A} \,\, holds \}.
\ee

Direct consequences are:
\bel{eq:BA-v-A}
\aligned
\big\|\del^I L^J v \big\|_{L^2_f(\Hcal_s)} + \big\| (s/t) \del \del^I L^J v, (s/t) \del \del^I L^J w \big\|_{L^2_f(\Hcal_s)}
&\lesssim C_1 \eps s^\delta,
\quad
&|I| + |J| \leq N,
\\
\big\|\del^I L^J v \big\|_{L^2_f(\Hcal_s)} + \big\| (s/t) \del \del^I L^J v \big\|_{L^2_f(\Hcal_s)}
&\lesssim C_1 \eps,
\quad
&|I| + |J| \leq N-3,
\\
|\del^I L^J v |
&\lesssim C_1 \eps t^{-1} s^\delta,
\quad
&|I| + |J| \leq N-2,
\\
\big|\del^I L^J v \big|
&\lesssim C_1 \eps t^{-1},
\quad
&|I| + |J| \leq N-5.
\endaligned
\ee

The estimates
$$
\big|\del^I L^J v \big|
\lesssim C_1 \eps t^{-1}
\lesssim C_1 \eps (s/t)^2,
\qquad
|I| + |J| \leq N-5
$$
as well as Lemma \ref{lem:equav} imply that
\bel{eq:A-equa2} 
E (s, \del^I L^J v, \del^I L^J w)
\lesssim
E_1 (s, \del^I L^J v ) + E (s, \del^I L^J w)
\lesssim
E (s, \del^I L^J v, \del^I L^J w),
\quad
|I| + |J| \leq N.
\ee

\subsection*{Improved estimates}

To improve the estimates appearing in the bootstrap assumptions \eqref{eq:BA-A}, we go through the analysis in Subsection \ref{subsec:refine}. To treat the quasilinear system \eqref{eq:model-q-A} our strategy is to apply the energy estimates \eqref{eq:EE-q-A1} for the high order energies, while we rely on the energy estimates \eqref{eq:w-EE}, \eqref{eq:w-EE2}, and \eqref{eq:ghost-D} for the low order energies, where we pretend the system \eqref{eq:model-q-A} is a semilinear system.

We start with the high order energy estimates of $v, w$ components.

\begin{lemma}
Under the assumptions in \eqref{eq:BA-A}, for all $s \in [s_0, s_1)$ we have
\bel{eq:A-w}
\aligned
E_1 (s, \del^I L^J v)^{1/2} + E (s, \del^I L^J w)^{1/2}
&\leq \eps + (C_1 \eps)^{3/2} s^\delta,
\qquad
&|I| + |J| \leq N.
\endaligned
\ee

\end{lemma}

\begin{proof}
We act $\del^I L^J$ on the equations in \eqref{eq:model-q-A} to have
$$
\aligned
&-\Box \del^I L^J v + \del^I L^J v + Q^{\alpha \beta} \del_\alpha \del_\beta \del^I L^J w = f_1,
\\
&-\Box \del^I L^J w + Q^{\alpha \beta} \del_\alpha \del_\beta \del^I L^J v = g_1,
\endaligned
$$
with
$$
\aligned
f_1 = & Q^{\alpha \beta} \del_\alpha \del_\beta \del^I L^J w - \del^I L^J \big( Q^{\alpha \beta} \del_\alpha \del_\beta w \big),
\\
g_1 = & Q^{\alpha \beta} \del_\alpha \del_\beta \del^I L^J v - \del^I L^J \big( Q^{\alpha \beta} \del_\alpha \del_\beta v \big).
\endaligned
$$
Then the energy estimates \eqref{eq:EE-q-A1} and \eqref{eq:A-equa2} deduce
$$
\aligned
&E_1 (s, \del^I L^J v) + E (s, \del^I L^J w)
\\
\lesssim
&E_1 (s_0, \del^I L^J v) + E (s_0, \del^I L^J w)
+
\int_{s_0}^s \int_{\Hcal_{s'}} (s'/t) \Big( f_1 \del_t \del^I L^J v + g_1 \del_t \del^I L^J w 
\\
&\hskip1cm
+ \del_\alpha Q^{\alpha \beta} \del_\beta \del^I L^J  v \del_t \del^I L^J  w + \del_\alpha Q^{\alpha\beta} \del_\beta \del^I L^J  w \del_t \del^I L^J  v 
 - \del_t Q^{\alpha\beta} \del_\beta v \del_\alpha w \Big) \, dxds'.
\endaligned
$$
We only provide the estimates for the term involving $f_1$, and other terms can be bounded in a similar way.
We proceed 
$$
\aligned
\int_{\Hcal_{s}} (s/t)
\big| f_1 \del_t \del^I L^J v \big| \, dx
\lesssim
&\sum_{\substack{|I_1| + |J_1| \leq N-5 \\|I_2| + |J_2| \leq N}} \Big( \big\| (t/s) \del^{I_1} L^{J_1} v \big\|_{L^\infty(\Hcal_s)}   \big\| (s/t) \del \del^{I_2} L^{J_2} w \big\|_{L^2_f (\Hcal_s)} 
\\
+ & \big\| \del^{I_2} L^{J_2} v \big\|_{L^2_f(\Hcal_s)}   \big\| \del \del \del^{I_1} L^{J_1} w \big\|_{L^\infty (\Hcal_s)} \Big) 
\big\| (s/t) \del_t \del^I L^J v \big\|_{L^2_f (\Hcal_s)} 
\\
\lesssim (C_1 \eps)^3 s^{-1+2\delta}.
\endaligned
$$
We thus have
$$
E_1 (s, \del^I L^J v) + E (s, \del^I L^J w)
\lesssim
\eps^2 + (C_1 \eps)^3 s^{2\delta}.
$$ 

The proof is done.
\end{proof}

Next, we turn to prove the refined pointwise decay estimates of $\del \del w$.

\begin{proposition}
We have
\be 
\aligned
\big\| (s/t) (t-|x|)^{1/2} \del \del \del^I L^J w \big\|
\lesssim 
\eps + (C_1 \eps)^{3/2},
\qquad
|I| + |J| \leq N-3,
\\
\big| \del \del \del^I L^J w \big|
\lesssim 
\big( \eps + (C_1 \eps)^{3/2} \big) s^{-1} (t-|x|)^{-1/2},
\qquad
|I| + |J| \leq N-5.
\endaligned
\ee

\end{proposition}

\begin{proof}
The proof is very similar to the proof of \eqref{eq:n12-5}, which follows from \eqref{eq:ddu} and
$$
\int_{\Hcal_s} (t-r)^{-4\delta} (s/t)^2 \big| \del \del^I L^J w \big|^2 \, dx
\lesssim
\eps^2 +  (C_1 \eps)^3,
\qquad
|I| + |J| \leq N-2,
$$
and we omit the details.
\end{proof}

Finally, we show the refined uniform energy bounds for $v$.

\begin{proposition}
It holds
\be 
E_1 (s, \del^I L^J v)^{1/2}
\leq \eps + (C_1 \eps)^{3/2},
\qquad
|I| + |J| \leq N-3.
\ee
\end{proposition}

\begin{proof}
The proof is very similar to the proof of Proposition \ref{prop:refine}, and we omit it.
\end{proof}

By carefully choosing $C_1$ large enough, and $\eps$ sufficiently small, we obtain the following refined estimates for $s \in [s_0, s_1)$ 
$$
\aligned
E_1 (s, \del^I L^J v)^{1/2} + E (s, \del^I L^J w)^{1/2}
&\leq {1\over 2} C_1 \eps s^\delta,
\qquad
&|I| + |J| \leq N,
\\
E_1 (s, \del^I L^J v)^{1/2}
&\leq {1\over 2} C_1 \eps,
\qquad
&|I| + |J| \leq N-3,
\\
\big| \del \del \del^I L^J w (t, x) \big|
&\leq {1\over 2} C_1 \eps s^{-1} (t-r)^{-1/2},
\qquad
&|I| + |J| \leq N-5,
\endaligned
$$
which imply the global existence result in Theorem \ref{thm:main2}. Together with the estimates in \eqref{eq:BA-v-A}, the pointwise decay results in Theorem \ref{thm:main2} are also proved.





{\footnotesize

\end{document}